\documentclass[twoside]{amsart}
\usepackage{amsaddr}

\usepackage{graphicx}
\usepackage{graphics}
\usepackage{hyperref}
\usepackage{xifthen}

\usepackage[inline]{enumitem}
\setenumerate{leftmargin=*}

\usepackage{amsmath,amssymb}
\usepackage{amsthm}
\usepackage{tikz-cd}
\usepackage{braket}
\usepackage{mathtools}
\usepackage{mathrsfs}

\theoremstyle{plain}
\newtheorem{theorem}{Theorem}[section]

\newtheorem{lemma}[theorem]{Lemma}
\newtheorem*{lemma*}{Lemma}
\newtheorem{proposition}[theorem]{Proposition}

\newtheorem*{corollary*}{Corollary}
\newtheorem{conjecture}[theorem]{Conjecture}
\newtheorem{claim}[theorem]{Claim}
\newtheorem*{claim*}{Claim}
\newtheorem*{theorem*}{Theorem}

\newtheoremstyle{break}%
{}{}%
{\itshape}{}%
{\bfseries}{}
{\newline}{}

\theoremstyle{break}

\theoremstyle{definition}
\newtheorem{definition}[theorem]{Definition}

\theoremstyle{remark}

\newtheorem{obs*}[theorem]{Observation}
\newtheorem*{remark}{Remark}
\newtheorem{remark*}[theorem]{Remark}

\newcommand{\N}{\mathbb{N}}
\newcommand{\Z}{\mathbb{Z}}

\newcommand{\itembf}[2]{%
\ifthenelse{\isempty{#1}}{\item {\bfseries #2}}{\item[\textbf{#1}]{\bfseries #2}}%
}

\newcommand{\dfn}{\mathrel{\mathop:}=}

\newcommand{\sqed}{\hfill $/\!\!/$}

\newcommand{\scratch}[1]{{\color{red}\textbf{#1}}}

\DeclareMathAlphabet{\mathbbold}{U}{bbold}{m}{n}
\DeclareMathAlphabet{\cmb}{OT1}{cmbr}{m}{n}

\DeclareMathOperator{\diam}{diam}

\DeclareMathOperator{\Mod}{Mod}
\DeclareMathOperator{\PMod}{PMod}

\DeclareMathOperator{\asdim}{asdim}

\newlist{clist}{enumerate}{1}
\setlist[clist]{label=(\roman*),itemsep=0ex,
  topsep=.5ex,parsep=0ex,leftmargin=3.5em}

\newlist{citem}{itemize}{1}
\setlist[citem,1]{label=\textbullet,topsep=1ex,itemsep=1ex,parsep=0ex}

\ifdefined\clean
\renewcommand{\scratch}[1]{{}}
\else
\fi

\newcommand{\A}{\mathcal A}
\newcommand{\G}{\mathcal G}
\newcommand{\cc}{\mathcal C}

\newcommand{\M}{\mathcal M}
\newcommand{\X}{\mathscr X}
\newcommand{\xc}{{\vphantom{\mathscr X}\smash{\hat{\mathscr X}}}}
\newcommand{\lx}{\mathcal L}

\DeclareMathOperator{\base}{base}
\DeclareMathOperator{\trans}{trans}

\begin{document}

\title{Multiarc and curve graphs are hierarchically 
hyperbolic}
\author{Michael C.\ Kopreski}
\email{kopreski@math.utah.edu}

\begin{abstract}
  A multiarc and curve graph is a simplicial graph whose vertices
  are arc and curve systems on a compact, connected, orientable
  surface $\Sigma$.
  We show that all connected, non-trivial 
  multiarc and curve graphs 
  preserved by the natural action of $\PMod(\Sigma)$ and 
	whose adjacent vertices have bounded geometric intersection
	number are hierarchically hyperbolic spaces with respect to
	witness subsurface projection.  This result extends work of
	Kate Vokes on twist-free multicurve graphs and confirms two
	conjectures of Saul Schleimer in a broad setting.  In addition, we
  prove that the $\PMod(\Sigma)$-equivariant quasi-isometry type of
  such a graph is uniquely classified by its set of 
  connected witness subsurfaces. 
\end{abstract}

\maketitle
\vspace{-1em}

\section{Introduction and main results}

Let $\Sigma$ be a compact, connected,
orientable, non-pants surface possibly with boundary such that
$\chi(\Sigma) \leq -1$. 
 We consider a broad
class of natural combinatorial models built from collections of 
arc and curve systems on $\Sigma$:
\begin{definition}
  A \textit{multiarc and curve graph} $\A$
  on $\Sigma$ is a
  simplicial graph whose vertices are collections of
  disjoint isotopy classes of simple, essential arcs and curves in
  $\Sigma$. 
  $\A$ is a \textit{multicurve graph} if $V(\A)$
consists of only multicurves.
\end{definition}

\begin{definition}\label{def:admis} 
  A connected multiarc and curve graph $\A$ on $\Sigma$ is
  \textit{admissible} if
  \begin{enumerate}[label=(\roman*)]
    \item \label{item:admis:nat_act} $\PMod(\Sigma)$ 
      preserves $V(\A)$ 
      and extends to an action on $\A$, and
    \item if $(a,b)\in E(\A)$, then the geometric
      intersection $i(a,b)$ is uniformly bounded.
  \end{enumerate}
\end{definition}

Recall that a \textit{witness subsurface} for $\A$ is an essential,
non-pants subsurface that meets every vertex of
$\A$.  We prove the following:

\begin{theorem}\label{thm:hhsa}
  Let $\A$ be an admissible multiarc and curve graph on $\Sigma$
  and let
  $\mathscr{X}$ denote its collection of witnesses.
  Then $(\A,\mathscr{X})$ is a hierarchically hyperbolic
  space with respect to subsurface projection to witness curve 
  graphs $\cc W$, $W \in \mathscr{X}$. 
\end{theorem}

\begin{theorem}\label{thm:a_class}
  The $\PMod(\Sigma)$-equivariant quasi-isometry type of an 
  admissible
  multiarc and curve graph on $\Sigma$ is uniquely determined by
  its connected witnesses. 
\end{theorem}

We first show the same for 
\textit{partial marking graphs}, whose vertex
set consists of markings in the sense of \cite{MMII} which are
``locally'' complete and clean, in Section~\ref{sec:hhsm}. 
Admissible partial marking graphs include
the $1$-skeleton $\mathcal{MC}^{(1)}(\Sigma)$ of Masur and Minsky's 
marking complex. Here, Theorem~\ref{thm:a_class} defines
a bijection: any set of
essential, non-pants subsurfaces closed under
enlargement and the action of $\PMod(\Sigma)$ is the witness set for
some admissible partial marking graph. 
In Section~\ref{sec:multiarc}, we
show that any admissible multiarc and curve graph on $\Sigma$ is
$\PMod(\Sigma)$-equivariantly quasi-isometric to an admissible
partial marking graph on $\Sigma$, via a map 
coarsely preserving subsurface projection, whence 
Theorems~\ref{thm:hhsa} and \ref{thm:a_class} follow.

The present work is indebted to the excellent paper
by Kate Vokes on the hierarchical hyperbolicity of 
\textit{twist-free} multicurve graphs \cite{vokes}, from which our
results borrow both inspiration and overall strategy.

\subsection{Background}

Given $\Sigma = \Sigma_g^b$ a compact, connected,
orientable surface with genus
$g$ and $b$ boundary components, let 
$\xi(\Sigma) = 3g + b - 3$ denote the \textit{complexity} of
$\Sigma$, equal to 
the number of components in a maximal multicurve.
For a multiarc and curve graph $\A$ on a compact surface $\Sigma$,
we define witness subsurfaces 
in the usual way:
\begin{definition}
  A compact, essential ($\pi_1$-injective, 
  non-peripheral) subsurface
  $W \subset \Sigma$ is a \textit{witness} for $\A$ if $W
  \not \cong \Sigma_0^3$ and every arc and curve system
  in $V(\A)$ intersects $W$.
\end{definition}

\noindent
Let $\pi_W : \A \to
2^{\cc W} \setminus \{\varnothing\}$ denote subsurface projection
and for $a,b \in V(\A)$, let $d_W(a,b) 
\dfn \diam_{\cc W} (\pi_W(a)\cup\pi_W(b))$.
Note that we do not require witnesses to be connected: 
if $W = V \sqcup V'$, let $\cc W$ be the graph join $\cc V *\cc V'$.

We consider two conjectures of Saul Schleimer \cite[\S 2.3]
{curvenotes} on the geometry of multiarc and curve graphs: the first
characterizes $\delta$-hyperbolicity, and the second proposes a
distance formula in the sense of \cite{MMII}.

\begin{conjecture}[Schleimer]
  \label{conj:witness} 
Let $M \geq 0$ and $\A$ be a 
multiarc and curve graph such that $(a,b)\in E(\A)$ if
and only if $i(a,b) \leq M$. $\A$ 
is $\delta$-hyperbolic if and only if it does not
admit disjoint connected witnesses.
\end{conjecture}

\begin{conjecture}[Schleimer]
  \label{conj:mm_dist} 
Let $M \geq 0$ and $\A$ be a 
multiarc and curve graph such that $(a,b)\in E(\A)$ if
and only if $i(a,b) \leq M$, and let $\mathscr{X}$ denote the
collection of witnesses of $\A$. Then there exists $C' = C'(\A)$ 
such that for any $C > C'$, there are constants $K,E \geq 0$ such 
that
\[
  d_\A(a,b) \stackrel{K,E}{=} \sum_{W \in \mathscr X} 
  [d_W(a,b)]_C\:.
\]
\end{conjecture}

If a multiarc and curve graph $\A$ on $\Sigma$
satisfies \ref{item:admis:nat_act} of Definition~\ref{def:admis},
then we will say that
$\A$ has a \textit{natural} $\PMod(\Sigma)$ action.
We note that Schleimer does not assume a natural action of
$\PMod(\Sigma)$, and indeed there exist interesting complexes
lacking such an action.  
Nonetheless, assuming a natural $\PMod(\Sigma)$ action and the
connectivity of $\A$ (as we will for
the remaining work), we observe that it is equivalent
in the conjectures
above to require only
that $i(a,b)$ be uniformly bounded for $(a,b) \in E(\A)$. In
particular:

\begin{remark*}\label{rmk:add_bdd_edges}
  Let $\A$ be a connected 
  multiarc and curve graph
  with a natural $\PMod(\Sigma)$ action, 
  and suppose $\A'$ is obtained by adding edges between
  vertices $a,b$ such that $i(a,b) < M$.  Then $\A
  \underset{\text{q.i.}}{\cong} \A'$.
\end{remark*}

Rather than approaching Conjectures~\ref{conj:witness} and
\ref{conj:mm_dist} directly, we appeal to a broader geometric
property: if $\A$ is hierarchically hyperbolic with the usual 
witness subsurface projection structure, then
both conjectures follow immediately. 
We will introduce hierarchical hyperbolicity and 
make precise these arguments at the end of the section.

\subsubsection{Low-complexity cases}

Note that $\xi(\Sigma) > 0$ by assumption.
When $\xi(\Sigma) \leq 0$, any multiarc and curve graph on $\Sigma$
is empty or it is finite and $\Sigma$ admits no essential
(non-peripheral) non-pants subsurfaces,
except when $\Sigma \cong \Sigma_{1}^0$.  In the latter case, any
admissible multiarc and curve graph is quasi-isometric to the curve
graph and $\Sigma_1^0$ is the only witness, hence
Theorems~\ref{thm:hhsa} and \ref{thm:a_class} hold 
trivially.  

\subsubsection{Twist-free multicurve graphs}

In the sense above, Conjectures~\ref{conj:witness} and
\ref{conj:mm_dist} were resolved positively for a broad class of
examples by Kate Vokes in \cite{vokes}.

\begin{definition}[Vokes]\label{def:twist-free}
  A multicurve graph $\G$ on $\Sigma$ is \textit{twist-free} if it
  is admissible\footnote{Vokes actually requires a natural action of
    $\Mod(\Sigma)$; however, any subgroup containing
  $\PMod(\Sigma)$ suffices for the arguments in
\cite{vokes}.}
  and admits no annular witnesses.
\end{definition}

\begin{theorem}[Vokes]\label{thm:tf_hhs}
Let $\mathcal G$ be a twist-free multicurve graph.  Then
$(\mathcal{G},\mathscr{X})$ is a hierarchically hyperbolic space
with respect to subsurface projections to witness curve graphs $\cc
W$, $W \in \mathscr{X}$.
\end{theorem}

\noindent
Theorem~\ref{thm:a_class} and in fact a 
bijection analogous to Theorem~\ref{thm:qi_marking} below 
follow easily from Vokes'
results in the special case of twist-free multicurve graphs.

The existence of annular witnesses appears related to
non-hyperbolicity, or equivalently by 
Conjecture~\ref{conj:witness}, the existence of disjoint connected
witnesses.  In particular, if $\xi(\Sigma) > 1$ and 
a multiarc and curve graph $\A$ on $\Sigma$
admits an annular witness $A$, then $\A$ admits a disjoint pair of
connected witnesses, namely $A$ and a complementary component.
As a partial converse, if $\A$ is an arc and curve graph that
admits two disjoint witnesses $W, W'$ separated by a 
unique complementary component, then $\A$ admits an
annular witness.
For example, given $\Gamma$ a relation on
$\pi_0(\partial\Sigma)$ the \textit{$\Gamma$-prescribed arc
graph} $\A(\Sigma,\Gamma)$, 
defined by the author in \cite{pac}, is the full subcomplex of
$\A(\Sigma)$ spanned by arcs between boundary components in 
$\Gamma$.  By passing to projections to $\cc\Sigma$,
$\A(\Sigma,\Gamma)$ may be viewed as an admissible multicurve graph;
however, whenever $\Gamma$ is
bipartite (in particular, whenever $\A(\Sigma,\Gamma)$ is
non-hyperbolic), it may be seen that 
$\A(\Sigma,\Gamma)$ admits $g - 1$ pairwise disjoint annular 
witnesses where $g$ is the genus of $\Sigma$, hence 
Theorem~\ref{thm:tf_hhs} does not apply.

Nonetheless, 
from Theorem~\ref{thm:hhsa} it follows immediately that
$\A(\Sigma,\Gamma)$ is hierarchically hyperbolic with respect to the
usual witness subsurface projection structure.  
As an immediately application, 
let $\Omega$ be an orientable surface of infinite topological type.
We recall the \textit{grand arc graph} $\mathcal{G}(\Omega)$ defined
in \cite{grand}, whose vertices are arcs between ends
of distinct maximal type, in the sense of the partial order given in
\cite{mannrafi}.  Suppose that $\Omega$ has exactly 
$n$ distinct maximal end classes.  Then for an arbitrarily large
witness $W \subset \Omega$ there exists a
$n$-partite relation $\Gamma_W$ on $\pi_0(\partial W)$
such that, by extending arcs,
$\A(W,\Gamma_W)$ quasi-isometrically embeds in 
$\mathcal{G}(\Omega)$.  Hence choosing an exhaustion of $\Omega$ by
connected witnesses, we have the following:

\begin{proposition}\label{prop:ga_rank}
  Suppose that $\Omega$ has exactly $2$ distinct maximal end
  classes and infinite genus.  
  There exists an asymphoric 
  hierarchically hyperbolic space of arbitrarily large rank 
  (see below)
  that quasi-isometrically embeds into $\mathcal{G}(\Omega)$. \qed
\end{proposition}

\noindent
More generally, the proposition also holds whenever $\Omega$ has
infinitely many non-maximal ends.
By \cite[Thm.~1.15]{quasiflats}, 
it follows that $\asdim \mathcal{G}(\Omega) = \infty$
whenever $\Omega$ has exactly two maximal end classes and infinite
genus or infinitely many non-maximal ends; we note this
fact may be likewise shown by explicitly constructing quasi-flats.

\subsubsection{Hierarchical hyperbolicity}

Behrstock, Hagen, and Sisto introduced hierarchically
hyperbolic spaces in \cite{hhsi} to provide a common geometric 
framework within which to study mapping class groups and cubical 
groups.  In effect, hierarchical hyperbolicity provides an
axiomatization and generalization of the geometric 
structure of the mapping class groups of finite type surfaces
elucidated in
\cite{MMII}; we direct the reader to \cite{hhsii} for a full
enumeration of the axioms, and to \cite{sisto} for a non-technical 
discussion of hierarchical hyperbolicity and a survey of 
results.

We provide a brief description.  Let $X$ be a quasi-geodesic 
space.  A hierarchically hyperbolic structure $(X,\mathcal{G})$ on
$X$ is comprised of the following:

\begin{itemize}
  \item an index set $\mathcal{G}$ and a collection of uniformly
$\delta$-hyperbolic spaces $\{\cc W : W \in \mathcal{G}\}$ with
associated projections $\pi_W : X \to 2^{\cc W} \setminus
\{\varnothing\}$;
\item relations $\sqsubseteq$ (\textit{nesting}) and $\perp$
(\textit{orthogonality}) on $\mathcal{G}$;
\item if $U \sqsubseteq V$ then a map $\rho^V_U : \cc V \to 2^{\cc
  U}$ and if additionally $U \neq V$, 
  a set $\rho_V^U \subset \cc V$.  If not $U
  \perp V$ and $U, V$ are not $\sqsubseteq$-comparable, 
  then a set $\rho_V^U \subset \cc V$ and
  \textit{vice versa}.
\end{itemize}

\noindent
In addition, the above must satisfy nine axioms for
$(X,\mathcal{G})$ to be hierarchically hyperbolic.  We
consider the model example: the marking complex
$\mathcal{MC}(\Sigma)$ defined in \cite{MMII},
on which $\Mod(\Sigma)$ acts geometrically,
is hierarchically hyperbolic with respect to the index set
$\mathcal{G} = \{\text{essential subsurfaces}\}$ and
subsurface projections $\pi_W$, with nesting and orthogonality
determined by inclusion and disjointness, respectively; $\rho_U^V$
is determined by subsurface projection, or the 
projection of boundary components, as appropriate
\cite[Thm.~11.1]{hhsii}. 

By defining a hierarchically hyperbolic structure analogously for
partial marking graphs, many of the axioms follow from the
fact that they hold for $\mathcal{MC}(\Sigma)$.  We omit these here
and enumerate two that we will check explicitly:
\begin{itemize}
  \item[\textbf{Ax.\ 1:}] \textit{(Projections.)} For $W \in
    \mathcal{G}$, $\pi_W$ must be uniformly coarsely Lipschitz and
    have uniformly quasi-convex image. 
  \item[\textbf{Ax.\ 9:}] \textit{(Uniqueness.)} For every 
    $K \geq 0$, there exists $K'\geq 0$ 
    such that if $u,v \in X$ such that $d_W(u,v)
    \leq K$ for all $W \in \mathcal{G}$, then $d(u,v) \leq K'$.
\end{itemize}

The \textit{rank} $\nu$ of a hierarchically hyperbolic space 
$(X,\mathcal{G})$ is the largest 
cardinality of a collection of pairwise orthogonal
elements $\{U_j\}$ in $\mathcal{G}$ for which $\pi_{U_j}(X)$ is
unbounded; if $(X,\mathcal{G})$ is \textit{asymphoric}, then there
exists $C \geq 0$ such that $\nu$ is likewise the maximum
cardinality of a pairwise orthogonal set for which
$\diam \cc U_j > C$.  $\delta$-hyperbolic spaces are hierarchically 
hyperbolic with respect to a trivial (hence rank at most $1$) 
hierarchical structure; conversely, any rank $1$ hierarchically
hyperbolic space is $\delta$-hyperbolic if it is also asymphoric
\cite[Cor.~2.15]{quasiflats}.  

For $\A$ an admissible multiarc and curve graph and 
a witness $W$, the action of $\PMod(\Sigma)$ implies 
$\pi_W(\A)$ is either unbounded or $\cc W$ is diameter at most
$2$, if $W$ is disconnected.  Hence Theorem~\ref{thm:hhsa} implies
that admissible multiarc and curve graphs are asymphoric
hyperbolically 
hierarchical spaces, and the above and that orthogonality is
equivalent to disjointness  implies that
Conjecture~\ref{conj:witness} holds. 

Let $d_W(a,b) \dfn \diam_{\cc W}(\pi_W(a) \cup \pi_W(b))$.
As with $\mathcal{MC}(\Sigma)$, hierarchically hyperbolic spaces
have a distance formula \cite[Thm.~4.5]{hhsii}:
\begin{theorem}[Behrstock-Hagen-Sisto]\label{thm:hhs_dist}
  Let $(X,\mathcal{G})$ be hierarchically hyperbolic.  Then there
  exists $C'$ such that for any $C > C'$, there exist $K,E
  \geq 0$ such that
  \[
    d(a,b) \stackrel{K,E}{=} \sum_{W \in \mathcal{G}}[d_W(a,b)]_C\:.
  \]
\end{theorem}

\noindent
Conjecture~\ref{conj:mm_dist} follows in the admissible case
from Theorem~\ref{thm:hhs_dist} and Theorem~\ref{thm:hhsa}. 
\phantom{.}\sqed

\section{Hierarchical hyperbolicity of partial marking
graphs}\label{sec:hhsm}

We begin with a minor generalization of the complete, clean
markings defined in \cite{MMII}.
\begin{definition}\label{def:marking}
  A \textit{marking}
  $\mu = \{(a_i,t_i)\}$ on $\Sigma$ is an essential
  simple multicurve $\{a_i\}$, denoted $\base \mu$, along with a
  collection of (possibly empty) diameter $1$ subsets $t_i \subset
  \cc (a_i)$, denoted $\trans \mu$; for $a_i \in \base \mu$, let
  $\trans(\mu, a_i) = t_i$ denote the associated 
  \textit{transversal.}
\end{definition}

\noindent
We say
$b$ is a \textit{clean transverse curve} for a curve $a$ 
if the subsurface $F$ filled by $a \cup b$ has
complexity $\xi = 1$ and $a,b$ are adjacent in $\cc F$.
\begin{definition}
  A marking $\mu$ is \textit{locally complete} if
  whenever $t_i \neq
  \varnothing$ then the complementary component of $\base \mu
  \setminus \{a_i\}$ containing $a_i$ has complexity $\xi = 1$; if
  in addition $t_i \neq
  \varnothing$ implies 
  $t_i = \pi_{a_i} b_i$ for some clean transverse curve
  $b_i$ intersecting $\base \mu$ only in $a_i$, 
  then $\mu$ is \textit{(locally) clean}.
\end{definition}

\begin{remark*}\label{rmk:loc_clean} 
  A locally clean marking is exactly one whose restriction to the
  maximal subsurface intersecting only components with non-empty
  transversals is a complete, 
  clean marking on that subsurface, 
  in the original sense of \cite{MMII}.  When we say a
  marking is
  \textit{clean}, we will always mean locally clean.
\end{remark*}

Just as with complete markings, given a locally complete marking
$\mu$, a
locally clean marking $\mu'$ is \textit{compatible} with $\mu$ if
$\base \mu = \base \mu'$ and, for all $a \in \base \mu$,
$\trans(\mu',a) = \varnothing$ if and only if $\trans(\mu,a) =
\varnothing$ and
$d_a(\trans(\mu,a),\trans(\mu',a))$ is minimal among all possible
choices of transversal. 
We have the following from Remark~\ref{rmk:loc_clean} and
\cite[Lem.\ 2.4]{MMII}:

\begin{lemma}\label{lem:marking_compat}
  For any 
  locally complete marking $\mu$, there exist at least one and at
  most $4^b$ compatible clean markings $\mu'$, 
  where $b$ is the number of 
  components in $\mu$ with non-empty transversal.  Furthermore, for
  any such $\mu'$ and $a \in
  \base \mu$, $d_a(\trans(\mu,a),\trans(\mu',a)) \leq 3$.
  \qed
\end{lemma}

\begin{definition}\label{def:geom_int}
  Let $\mu = \{(a_j,t_j)\}, \nu=\{(b_k,s_k)\}$ be two markings on
  $\Sigma$.  Then their \textit{geometric
  intersection number} $i(\mu,\nu)$ is defined as follows:
  \begin{align*}
    i(\mu,\nu) &\dfn i(\base \mu,\base \nu)  \\
                 & \qquad + \sum_j 
                 \diam_{\cc(a_j)}\left[\pi_{a_j} (\base \nu) 
                 \cup t_j\right] 
                 + \sum_k 
                 \diam_{\cc(b_k)}\left[\pi_{b_k} (\base \mu) 
                  \cup s_k\right] \\
                 & \qquad + \sum_{a_j = b_k} \diam_{\cc(a_j)}
                  (t_j \cup s_k)
  \end{align*}
\end{definition}

\begin{definition}\label{def:marking_graph}
  A \textit{partial marking graph} $\mathcal{M}$ on $\Sigma$ 
  is a simplicial
  graph whose vertices are clean markings.  If in addition
 \begin{enumerate}[label=(\roman*)]
   \item $\mathcal M$ is connected;
   \item has a natural $\PMod(\Sigma)$ action; and
   \item there exists $L \geq 0$ such that 
     $(\mu,\nu) \in E(\mathcal{M})$ only if $i(\mu,\nu) \leq L$,
 \end{enumerate} 
 then $\mathcal{M}$ is \textit{admissible.}
\end{definition}

We note that admissible multicurve graphs are likewise admissible
partial marking graphs; we need only endow each multicurve with
empty transversals.  The $1$-skeleton of the 
marking complex $\mathcal{MC}(\Sigma)$ in \cite{MMII} is 
an admissible partial marking graph: its vertex set consists of
complete clean markings.

\begin{definition}\label{def:marking_witness}
Let $\M$ be a partial marking graph on $\Sigma$.  
An essential subsurface $W \subset \Sigma$ is a \textit{witness} for
$\M$ if $W \not\cong \Sigma_0^3$ and for every marking $\mu \in
V(\mathcal{M})$ either 
\begin{enumerate*}[label=(\roman*)]
  \item $W$ intersects $\base \mu$ or 
  \item $W$ has an annular component whose core is a component in 
$\base \mu$ with non-empty transversal.
\end{enumerate*}
\end{definition}

Subsurface projection extends to markings, as defined in
\cite{MMII}:

\begin{definition}[Subsurface projection]\label{def:marking_proj}
  Let $\M$ be a partial marking graph on $\Sigma$, and let $S\subset
  \Sigma$ be a connected essential subsurface. 
  We define $\pi_S:\M \to 2^{\cc S}$ as
  follows: for $\mu \in V(\M)$,
  if $S$ is an annulus with core $\alpha \in \base \mu$, then
  $\pi_S(\mu) \dfn \trans(\mu,\alpha)$.
  Otherwise, $\pi_S(\mu) \dfn \pi_S(\base \mu)$ is the usual
  multicurve subsurface projection.  For $S$ the disjoint 
  union of connected essential subsurfaces $S_j$, 
  let $\pi_S(\mu) \dfn \bigcup_j \pi_{S_j}(\mu) \subset \cc S$. 
\end{definition}

\begin{remark*}\label{rmk:coarse_ss_proj}
  We note that $\pi_W(\mu) \neq \varnothing$ for all $\mu \in V(\M)$
  if and only if $W$ is a witness for $\M$.
  $\diam_{\cc W}(\pi_W(\mu)) \leq 2$ by \textit{e.g.}\
  \cite[Lem.~2.3]{MMII}.  For convenience, we shall denote
  $d_W(\mu,\nu) \dfn
  \diam_{\cc W}(\pi_W(\mu)\cup\pi_W(\nu))$.
\end{remark*}

\begin{theorem}\label{thm:hhsm}
  Let $\M$ be an admissible partial marking graph, and let 
  $\mathscr{X}$ denote its collection of witnesses.
  Then $(\M,\mathscr{X})$ is a hierarchically hyperbolic
  space with respect to subsurface projection to witness curve 
  graphs
  $\cc W$, $W \in \mathscr{X}$.
\end{theorem}

We follow the strategy in \cite{vokes}, extending as necessary to
our setting.  We first show that the quasi-isometry type of an
admissible partial marking graph $\M$ 
is fully determined by its set of
witnesses $\mathscr X$; in particular, there exists a canonical
``maximal'' representative $\M \xrightarrow{\text{q.i.}} 
  \mathcal{L}_{\mathscr{X}}$, which we then conclude to have the
  desired hierarchically hyperbolic structure.

\subsection{A universal partial marking graph}
  \label{sec:qi_marking}

We establish a bijection between the 
coarsely $\PMod(\Sigma)$-equivariant
quasi-isometry types of admissible partial marking graphs on
$\Sigma$ and certain
collections of connected essential
subsurfaces of $\Sigma$.

\begin{definition}\label{def:witness_set}
  A set $\mathscr{X}$ of essential, non-pants
  subsurfaces of $\Sigma$ is an
  \textit{admissible witness set} for $\Sigma$ 
  if it is closed under enlargement
  and the action of $\PMod(\Sigma)$.  Let the \textit{connected
  witness set} $\xc \subset \X$
  denote the connected subsurfaces in $\X$.
\end{definition}

The collection of witnesses for an
admissible partial marking graph $\M$ is an admissible 
witness set, denoted $\mathscr{X}^\M$. 
More generally, any collection of essential subsurfaces may be 
closed to a admissible witness set.  

\begin{theorem}\label{thm:qi_marking}
  The map $\M \mapsto \mathscr{X}^\M$
  induces a bijection $\Psi : [\M] \mapsto \xc^\M$ 
  between coarsely $\PMod(\Sigma)$-equivariant
  quasi-isometry types of admissible partial marking graphs and
  admissible connected witness sets.
\end{theorem}

That the map $\Psi$ is well defined
will follow from the fact that
witness subsurface projections of admissible marking graphs are 
Lipschitz, as we will see in Section~\ref{sec:lip_proj}. 
To show that $\Psi$ is bijective, we construct an admissible 
partial marking graph
$\mathcal{L}_{\xc}$ for any admissible witness set 
$\X$ such that 
\begin{enumerate*}[label=(\roman*)]
  \item $\hat \X^{\lx_\xc} = \hat \X$, and  
  \item for any admissible marking
graph $\M$ satisfying $\hat \X^\M = \hat \X$, there
is a coarsely $\PMod(\Sigma)$-equivariant quasi-isometry $\M
\to \mathcal{L}_{\xc^\M}$.
\end{enumerate*}

We extend the notion of an \textit{elementary move} on a clean
marking, as introduced in \cite{MMII}:
\begin{definition}
  Let $\mu$ be a clean marking on $\Sigma$, and suppose a clean
  marking $\mu'$ is
  obtained from $\mu$ by either
  \begin{enumerate}[label=(\roman*)]
    \item \textit{a twist:} replace some
      component $(a,\pi_a b) \in \mu$ with $(a, \pi_a \tau_a b)$,
      for $\tau_a$ a Dehn twist or a half Dehn twist
      about $a$.
    \item \textit{a flip:} replace some
      component $(a,\pi_a b) \in \mu$ with $(b, \pi_b a)$ and choose
      a clean marking compatible with the result.
\end{enumerate}
  Then $\mu'$ is obtained from $\mu$ by an \textit{elementary move}.
  (We assume $i(a,b) > 0$.)
\end{definition}

\begin{remark}
  While not canonical in a strict sense, by
  Lemma~\ref{lem:marking_compat} flip moves are unique up to
  finitely many choices for fixed $a,b$, 
  all uniformly close after projection to
  witnesses.
\end{remark}

\begin{definition}\label{def:lx}
  Let $\X$ be an admissible witness set on $\Sigma$.
  Define $\lx_\X$ to be the
  simplicial graph whose vertices are all clean markings that meet
  every surface in $\X$, in the sense of
  Definition~\ref{def:marking_witness}, and for which $(\mu,\nu) \in
  E(\lx_\X)$ if and only if $\mu$ is obtained from $\nu$ by either 
  \begin{enumerate}[label=(\roman*)]
    \item adding or removing a component $(a,t)$, or
    \item adding or removing a transversal, or
    \item an elementary move.
  \end{enumerate}
  Let $\lx_{\xc}$ be defined analogously, replacing
  $\X$ with $\xc$.
\end{definition}

\begin{lemma}\label{lem:a}
  Let $\X$ be an admissible witness set.  $\X^{\lx_\X} = \X$
  and $\xc^{\lx_\xc} = \hat \X$.
\end{lemma}

\begin{proof}
  $\X \subset \X^{\lx_\X}$
since by definition the vertices of $\lx_\X$ meet every subsurface
in $\X$, and likewise $\hat \X \subset \X^{\lx_\xc}$ hence 
$\hat \X \subset \hat \X^{\lx_\xc}$.  
Conversely, suppose $W \not\in \X$ is an essential,
non-pants subsurface of $\Sigma$.  $\X$ is closed under enlargement,
hence $W$ contains no subsurfaces in $\X$.  Let $\mu$ be the clean
marking on $\Sigma$ obtained from any complete, clean marking on
$\Sigma \setminus W$ by adding distinct components in $\partial W$
with empty transversals.  Then $\mu \in V(\lx_\X)$ and does not
meet $W$. $W \notin \X^{\lx_\X}$. Likewise, if $W \not\in \xc$
is an essential connected non-pants subsurface, then $W \not\in
\hat \X^{\lx_\xc}$.
\end{proof}

It is clear that $\PMod(\Sigma)$ acts naturally on $\lx_\X$ and
$\lx_\xc$ and that, in either graph, 
two markings are adjacent only if their geometric intersection is
less than some uniform constant $L$.
We prove below that $\lx_\X$ and $\lx_\xc$ are connected,
hence admissible.  In addition, 
$\lx_\xc$ is universal in the following sense: suppose that $\M$ is
an admissible marking graph such that $\X^\M \supset \xc$.  Then
every marking in $\M$ meets every witness in $\xc$, hence $V(\M)
\subset V(\lx_\xc)$ 
and is preserved by the action $\PMod(\Sigma)$; 
we show below that this inclusion is Lipschitz, hence induces
a $\PMod(\Sigma)$-equivariant 
coarse Lipschitz map $\hat \iota : \M \to \lx_\xc$. An identical
argument gives a $\PMod(\Sigma)$-equivariant coarse Lipschitz map
$\iota : \M \to \lx_\X$ whenever $\X^\M \supset \X$, in which case
$\hat \iota $
factors coarsely (exactly on vertices)
as the composition of the universal maps $\M \to \lx_\X \to \lx_\xc$
induced by the inclusions $\X^\M \supset \X \supset \xc$. 

We note that
$\mathcal{MC}(\Sigma)^{(1)}$ admits every essential, non-pants
subsurface as a witness, hence $\X(\mathcal{MC}(\Sigma)^{(1)})
\supset \X \supset \xc$.  Since the edges of 
$\mathcal{MC}(\Sigma)^{(1)}$ 
correspond to elementary moves, $\iota$ may be taken to be a
(simplicial) embedding and $\mathcal{MC}(\Sigma)^{(1)}$ as a
connected $\PMod(\Sigma)$-invariant full subgraph of $\lx_\X$, which
is in turn a $\PMod(\Sigma)$-invariant full subgraph of $\lx_\xc$. 
Hence to show that $\lx_\X$ and $\lx_\xc$ are connected,
it suffices to prove that every marking lies in the same component 
as $\mathcal{MC}(\Sigma)^{(1)}$, as the following lemma shows:

\begin{lemma}\label{lem:marking_completion}
 Any marking $\mu$ in $V(\lx_\X)$ or $V(\lx_\xc)$ 
 may be completed to a complete, clean marking $\mu'$ through a 
sequence of length at most $2\xi(\Sigma) - |\mu|$ 
of adding components and transversals.
\end{lemma}

\begin{proof}
  We note that we may always add components with empty transversal:
  since $\mu$ is already locally complete, no new base curve
  intersects existing clean transverse curves.  
  Likewise, we may always add (clean) transversals to complete
  markings.  Add $\xi(\Sigma) - |\mu|$
  components with empty transversal to obtain the complete marking
  $\mu''$. For at most $\xi(\Sigma)$ components with empty 
  transversal in $\mu''$, add transversals to obtain a 
  complete, clean marking $\mu' \in \mathcal{MC}(\Sigma)^{(1)}$.
\end{proof}

\subsubsection{The quasi-isometry}

We restate the arguments in \cite{vokes}, with a substantial
adapatation 
of the proof of the existence of a quasi-retraction for our setting.
Let $\M$ be an admissible partial marking graph on $\Sigma$ with
witness set $\X$, and let $\hat \iota : \M \to \lx_\xc$ be the universal
coarse map from above.  We show that $\hat \iota$ is a quasi-isometry.

We will 
use a key feature of admissibility, namely the existence of an
upper bound on distance in terms of geometric intesection number:

\begin{lemma}\label{lem:di_bound}
  Any admissible partial marking graph $\M$ admits a monotonic 
  function $f_\M
  : \N \to \N$ such that for all $\mu,\nu \in V(\M)$, $d(\mu,\nu)
  \leq f_\M(i(\mu,\nu))$.
\end{lemma}

\noindent
In fact, Lemma~\ref{lem:di_bound} follows only
from connectivity and the
natural action of a finite index subgroup $H \leq \Mod(\Sigma)$: for
any $M \geq 0$,
$\Mod(\Sigma)$ acts cofinitely on the set of pairs of clean markings
with geometric intesection less than $M$. 
However, given that $(\mu,\nu) \in E(\M)$ only if $i(\mu,\nu) \leq
L$, it is immediate that $\hat \iota$ is 
$f_{\lx(\xc)}(L)$-Lipschitz;
we remark that this argument
implies that the universal maps $\hat \iota : \M \to \lx_\xc$ and 
$\iota : \M \to \lx_\X$ are always Lipshitz, even if
$\X^\M \supsetneq \X$.

Coarse surjectivity likewise 
results from Lemma~\ref{lem:di_bound} and the
following:

\begin{lemma}\label{lem:iota_surject}
  There exists $M \geq 0$ such that 
  for any $\mu \in V(\lx_\xc)$, there exists $\nu \in V(\M)$ such
  that $i(\mu,\nu) \leq M$.
\end{lemma}

\begin{remark*}\label{rmk:iota_surject}
  Lemma~\ref{lem:iota_surject} is a 
  special case of following
  general fact: for any $G$-invariant function 
  $s : A \times B \to \N$ on cofinite $G$-sets $A,B$, $\min_{b\in
  B} s(a,b)$ is uniformly bounded for $a \in A$.  
  In fact, it follows that for any cofinite $G$-set $A$, there
  exists $C \geq 0$ such that for any $a, a' \in A$, there exists $g
  \in G$ such that $s(a,ga') \leq C$: let $C$ be the maximum of the
  bounds obtained for $B \in A/G$.  We will use this fact below.
\end{remark*}

Finally, we construct a coarse Lipschitz
retraction $\rho : \mu \mapsto E_\mu$,
where $E_\mu = \{\nu \in V(\M) : i(\mu,\nu) \leq M\}$ and $M$ is
the constant in Lemma~\ref{lem:iota_surject}.  By
Lemma~\ref{lem:di_bound}, it is clear that $\hat \iota\rho$ is coarsely
identity. Moreover, we note that $\PMod(\Sigma)$ acts cofinitely on
pairs of adjacent vertices $(\mu,\mu') \in E(\lx_\xc)$ since
$i(\mu,\mu') < L$, and that 
$\rho$ is $\PMod(\Sigma)$-equivariant, hence $d_\M(E_\mu,E_{\mu'})$
is uniformly bounded.
It remains only to check the following:

\begin{lemma}\label{lem:marking_retract}
  $\diam_\M(E_\mu)$ is uniformly bounded.
\end{lemma}

We must first build some machinery.  Let $K \subset \Sigma$ be an
essential, non-peripheral, and possibly disconnected subsurface.  
Given $\omega$ a clean
marking on $\Sigma$, fix a representative in minimal position with
$\partial K$ and let $\omega|_K \dfn 
(\sigma,\alpha)$, where $\sigma \subset \omega$ is
  the submarking comprised
  of components whose representative curve is 
  fully contained in $K$, and
  $\alpha$ is an arc system  
  obtained from the remaining components by taking the
  intersection of their representative curves
  with $K$.  We regard $\omega|_K$ up to isotopy, rel
  $\partial K$.  Let $i_{K}(\omega|_K,\omega'|_K) \dfn
    i(\sigma,\sigma')
  + i(\alpha,\sigma') + i(\sigma,\alpha') + i(\alpha,\alpha')$,
  where $\omega|_K = (\sigma,\alpha)$ and $\omega'|_K =
  (\sigma',\alpha')$ and the 
  geometric intersection number between arc systems
  and markings is defined as follows:
  \[
    i(\nu,\alpha) \dfn i(\base \nu,\alpha) + \sum_{(a,t) \in \nu}
    \diam_{\cc(a)}(t\cup\pi_a \alpha)
  \]
  We emphasize that $\omega|_K$ is well
  defined only up to choice of representative isotopic
  rel $\partial K$, which we assume to be fixed for given $\omega$ 
  unless otherwise specified.  
  $\Mod(K,\partial K)$ acts on the set
  of pairs $\omega|_K$ and preserves $i_{K}$.

\begin{remark}
  Here our notation differs with that of \cite{MMII}, where
  $\omega|_K$ instead denotes the \textit{restriction} 
  of $\omega$ to $K$.
\end{remark}

\begin{claim}\label{claim:phi_ikj} 
  There exists $D \geq 0$ such that 
  for any $\omega|_K,\omega'|_K$ with $i(\omega,\partial K),
  i(\omega',\partial K)$ at most $M$,
  there exists $\phi \in
  \Mod(K,\partial K)$ for which $i_{K}(\phi\omega|_K,
  \omega'|_K) \leq D$.
\end{claim}
  
\begin{proof}
  Let $\mathscr{S}$ be the set of equivalence classes of pairs
  $\eta|_K$ for some choice of representative of $\eta \in V(\M)$
  with $i(\eta,\partial K) \leq M$,
  up to (endpoint free) isotopy and ignoring
  transversals on components peripheral in $K$.  $\PMod(K)$ acts
  cofinitely on $\mathscr{S}$ and preserves 
  \[i_{K}([\eta|_K],[\eta'|_K]) \dfn \min_{\rho|_K \in [\eta|_K], 
    \rho'|_K \in [\eta'|_K]}i_{K}(\rho|_K,\rho'|_K)
  \]
  hence likewise does
  $\Mod(K,\partial K) \twoheadrightarrow \PMod(K)$.  By
  Remark~\ref{rmk:iota_surject}, there exists $D'$ independent of
  $\omega|_K, \omega'|_K$ and $\psi \in
  \Mod(K,\partial K)$ such that
  $i_{K}([\psi\omega|_K],[\omega'|_K]) \leq D'$. However, we
  observe that if $[\eta|_K] = [\eta'|_K]$, then there exists $D''$
  depending only on $M$ and some
  boundary multitwist $\tau\in\Mod(K,\partial K)$ such that 
  $i_{K}(\tau\eta|_K,\gamma|_K) \leq i_{K}(\eta'|_K,\gamma|_K) 
  + D''$ for all $\gamma|_K$, whence the
  claim follows with $D = D' + 2D''$.
\end{proof}

\begin{claim}\label{claim:ikj}
  Let $\eta$ be a multicurve on $\Sigma$ 
   and $K_l \subset \Sigma$ pairwise disjoint subsurfaces
   partitioning $\Sigma$ whose boundary curves lie in $\eta$.
 There exists $C \geq 0$ such that, for any markings
 $\omega, \omega'$ with $i(\omega,\eta),i(\omega',\eta) \leq M$,  
  $i(\omega,\omega') \leq \sum_l
  i_{K_l}(\omega|_{K_l},\omega'|_{K_l})+ C$.
\end{claim}

\begin{proof}
  Let $\gamma = \{(a_j,t_j)\}, \gamma'=\{(a'_k,t'_k)\}$ be 
  the maximal submarkings of $\omega,
  \omega'$ respectively whose base curves all
  intersect $\eta$.  Then by our
  definition of $i_{K_l}$, we have that 
    \begin{align*}
      i(\omega,\omega') \leq \sum_l i_{K_l}(\omega|_{K_l},
      \omega'|_{K_l}) 
                        &+ \sum_{(a_j,t_j) \in \gamma} 
                        \diam_{\cc (a_j)}(t_j\cup \pi_{a_j}
                        (\base \omega')) \\
                        &+ \sum_{(a'_k,t'_k) \in \gamma'} 
                        \diam_{\cc (a'_k)}(t'_k\cup \pi_{a'_k}
                        (\base \omega)) \\
                        &+ \sum_{a_j = a'_k}
                        \diam_{\cc (a_j)}(t_j\cup
                        t'_k)\:.
    \end{align*}  
  It suffices to show that the (at most $3\xi(\Sigma)$) 
  transversal terms are uniformly bounded.
  Suppose $(a_j,t_j) \in \gamma$.  $a_j$ intersects $\eta$,
  hence $\pi_{a_j} \eta \neq \varnothing$.  
  Since $i(\omega,\eta), i(\omega',\eta) \leq M$,
  $\diam_{\cc(a_j)}(t_j \cup \pi_{a_j} \eta)
  \leq M$ and 
  $\diam_{\cc(a_j)}(\pi_{a_j}\eta \cup \pi_{a_j}(\base\omega'))$
  is uniformly bounded,  
  hence $\diam_{\cc(a_j)}(t_j \cup \pi_{a_j}(\base \omega'))$ is
  uniformly bounded.  Analogous arguments apply for $(a'_k,t'_k)
  \in \gamma'$ and when $a_j = a'_k$.
  \end{proof}

\begin{proof}[Proof of Lemma~\ref{lem:marking_retract}]
  Fix a representative of $\mu$ and let $N$ be an open regular
  neighborhood, and let $K_j$ denote 
  the complementary components of $N$
  along with the components of $\overline N$. 
  Let $K^j \dfn \Sigma \setminus \mathring K_j$.
  Fix representatives of
  $\nu,\nu'$ in minimal position with $\partial K_j$ for all $j$;
  since $\nu,\nu' \in E_\mu$, the number of arc components in $\nu
  \cap K_j$ and $\nu' \cap K_j$ is at most $M$, hence 
  whenever $K_j \cong \Sigma_0^3$ 
  we may assume that
  $i_{K_j}(\nu|_{K_j},\nu'_{K_j})$ is uniformly bounded up to
  isotoping intersections into surrounding annuli. 
  Let $\nu_0 = \nu$.  For each $K_j$, we will construct 
  $\nu_j \in V(\M)$ such that $\nu_j$ is
  identical to $\nu_{j-1}$ except in $\mathring K_j$ 
  (hence $i(\nu_j,\base \mu) \leq M$) 
  and $d_\M(\nu_j,\nu_{j-1})$ and 
  $i_{K_j}(\nu_j|_{K_j},\nu'|_{K_j})$ are uniformly bounded.
  The lemma then follows from Claim~\ref{claim:ikj}.  

  We construct
  $\nu_j$ inductively: assume $\nu_{j-1}$ exists and suppose that
  $K_j \not\in \xc$,
  hence $K_j \not\in \X$ since $K_j$ is connected.  
  If $K_j \cong \Sigma_0^3$ then it suffices to
  let $\nu_j = \nu_{j-1}$.  Else, there exists $\omega \in
  V(\M)$ disjoint from $K_j$.  Applying Claim~\ref{claim:phi_ikj},
  up to translation by $\Mod(K^j,\partial K^j)$ we may assume that 
  $i_{K^j}(\omega|_{K^j},\nu_{j-1}|_{K^j}) \leq D$; likewise, choose
  $\phi \in \Mod(K_j,\partial K_j)$ such that
  $i_{K_j}(\phi\nu_{j-1}|_{K_j},\nu'|_{K_j}) \leq D$ and, extending 
  $\phi$ by identity, let $\nu_j = \phi\nu_{j-1}$.
  Then $i_{K_j}(\nu_j|_{K_j},\nu'|_{K_j}) \leq D$.  
  $\omega$ is disjoint  
  from $K_j$, hence $\omega|_{K_j} = \varnothing$ and
  $i_{K_j}(\nu_{j-1}|_{K_j},\omega|_{K_j}) = 0$ and by
  Claim~\ref{claim:ikj}, $i(\nu_{j-1},\omega) \leq D + C$. 
  $\nu_j, \nu_{j-1}$ are identical outside of $K_j$, thus likewise
  $i(\nu_j,\omega) \leq D + C$.  Hence
  $d_\M(\nu_j,\nu_{j-1}) \leq d_\M(\nu_j,\omega) +
  d_\M(\omega,\nu_{j-1}) \leq 2 f_\M(D + C)$.

  If $K_j \in \xc$, then $K_j$ is an annulus; let 
  $a_j \in \base \mu$ be its core
  curve. Since $K_j$ is a witness,
  $\trans(\mu,a_k) \neq \varnothing$.
  $i(\mu,\nu'), i(\mu,\nu_{j-1}) \leq M$ then implies that
  $i_{K_j}(\nu_{j-1}|_{K_j},\nu'|_{K_j})$ is uniformly bounded. 
  Let $\nu_j = \nu_{j-1}$.
\end{proof}

\begin{remark}
The universal map $\iota : \M \to \lx_\X$ is also a
$\PMod(\Sigma)$-equivariant quasi-isometry.  
In particular, the universal map $\iota':\lx_\X \to \lx_\xc$ is a
$\PMod(\Sigma)$-equivariant quasi-isometry by the above
 and $\hat \iota : \M
\to \lx_\xc$ factors through $\iota$ and $\iota'$. \sqed
\end{remark}

\subsubsection{Witness subsurface projections are
Lipschitz}\label{sec:lip_proj} 

We consider $\pi_W : \lx_\xc \to 2^{\cc W}$ for $W \in \xc$;
the arguments for $\lx_\X$ are identical.
It suffices that for $(\mu,\nu) \in E(\lx_\xc)$ and $W \in
\xc$, $d_W(\mu,\nu) = 
\diam_{\cc W}(\pi_W(\mu)\cup\pi_W(\nu))$ is uniformly bounded.
If $\mu,\nu$ differ  
by the addition of a component or a transversal, 
then without loss of generality $\pi_W(\mu) \subset
\pi_W(\nu)$ and $d_W(\mu,\nu) = 
\diam_{\cc W}(\pi_W(\nu)) \leq 2$. 
Otherwise, $\mu,\nu$ differ by an elementary move. 
Since $\mu,\nu$ are locally complete and clean, we may 
apply the proof of Lemma~2.5 in \cite{MMII}:

\begin{proposition}[Masur-Minsky]\label{prop:pi_lip}
  Suppose $\mu,\nu \in V(\lx_\xc)$ and $\nu$ is obtained from $\mu$
  by an elementary move.  Then for any $W \in \xc$, $d_W(\mu,\nu)
  \leq 4$. 
\end{proposition}

We may now prove that the map $[\M] \mapsto \xc^\M$ from
Theorem~\ref{thm:qi_marking} is well defined.  Let $\xc = \xc^\M$.
Since $V(\M) \subset
V(\lx_\xc)$, the extension $\hat \iota : \M \to \lx_\xc$
preserves subsurface projection; since $\hat \iota$ is a 
quasi-isometry, witness subsurface projections $\pi_W : \M \to
2^{\cc W}$, $W \in \xc$ are likewise Lipschitz.
Suppose that $\M, \M'$ have distinct connected 
witness sets $\xc, \xc'$
respectively, and let $W \in \xc' \setminus \xc$.  Then there exists
$\omega \in V(\M)$ disjoint from $W$.  Since $W$ is connected, we
may choose a loxodromic element 
$\varphi \in \Mod(W,\partial W)$ for $\cc W$, any extension of 
which acts on $\M'$ with
non-zero translation length since $\pi_W : \M' \to 2^{\cc
W}$ is Lipschitz. Let $\tilde\varphi \in \PMod(\Sigma)$ 
be an extension by identity.
Then $\tilde \varphi$ fixes $\omega \in \M$, hence $\M, \M'$ are not
$\PMod(\Sigma)$-equivariantly quasi-isometric. \sqed

\begin{remark}\label{rmk:mc_distort}
  The above implies that $\mathcal{MC}(\Sigma)^{(1)}$, while a
  $2\xi(\Sigma)$-coarsely dense full subcomplex of $\lx_\X$ by
  Lemma~\ref{lem:marking_completion}, 
  is in general significantly distorted.
\end{remark}

\subsection{Hierarchical structure of $(\mathcal{L}_{\mathscr{X}},
\mathscr{X})$}

We endow $\lx_\X$ with the usual hierarchical structure via witness
subsurface projections $\{\pi_W: \lx_\X \to 2^{\cc W}
\}_{W \in \X}$ and the following relations
on $\X$:

\begin{itemize}
  \item $U \sqsubseteq V$ if and only if 
    $U \subset V$ up to isotopy; and
  \item $U \perp V$ if and only if $U, V$ are disjoint up to
    isotopy.
\end{itemize}
Let $U \pitchfork V$ if neither of the above hold.  
For $W \in \X$, $\pi_W$ is uniformly quasi-convex: any curve
(potentially adding transversal) may be completed to a marking in
$V(\lx_\X)$, thus 
$\pi_W$ is surjective if $W$ is not an annulus; 
if $W$ is an annulus with core curve
$a$, then any orbit of a Dehn twist about $a$ projects to a 
$1$-coarsely dense subset of $\cc W$.
Hence $\pi_W$ satisfies Axiom 1 by the above, 
Remark~\ref{rmk:coarse_ss_proj}, and
Proposition~\ref{prop:pi_lip}.

If $U \sqsubseteq V$, then let $\rho_U^V
= \pi_U : \cc V \to 2^{\cc U}$, and if in addition 
$U \neq V$, let
$\rho_V^U = \partial_V U$, the non-peripheral boundary of $U$ in
$V$.  If $U \pitchfork V$, let $\rho_U^V =
\pi_U(\partial V)$. 
Axioms 2 and 3 follow immediately from the definition of witness
subsurfaces\footnote{We remark that Axiom 3 may require
disconnected witnesses. Thus $(\lx_\xc,\xc)$ may not be 
hierarchically hyperbolic, despite that the quasi-isometry
$\hat \iota : \lx_\X \to \lx_\xc$ preserves projections over $\xc$.}
and our choice of relations and
associated projections, and
Axioms 4 through 8 follow from the fact that the same hold for
$\mathcal{MC}(\Sigma)^{(1)}$, 
which lies as a $2\xi(\Sigma)$-coarsely
dense subcomplex of $\lx_\X$, and that the $\pi_W$
are uniformly Lipschitz.  
We verify Axiom 9
below.

\begin{definition}\label{def:xtilde}
  Given an admissible witness set $\X$, let the \textit{twist-free
  witness set} $\tilde\X \subset \X$ denote the
  admissible witness set with all annular witnesses removed.  
\end{definition}

\begin{definition}\label{def:kx}
  If $\X$ is an admissible witness set without annuli, let
  $\mathcal{K}_\X$ denote the full subgraph of $\lx_\X$ spanned by
  markings with empty transversals, with
  additional edges $(\mu,\nu)$ corresponding 
  to \textit{twist-free flip moves}:
  $\nu$ is obtained from $\mu$ by replacing a component
  $(a,\varnothing) \mapsto (b,\varnothing)$ such that, if $F$ is the
  subsurface filled by $a,b$, then $F$ is connected, $\partial F
  \subset \base \mu \cup \partial \Sigma$, and $d_{\cc F}(a,b) = 1$.
\end{definition}

\begin{lemma}\label{lem:axiom9}
 For any $K \geq 0$, there exists $K' \geq 0$ such that for any
 $\mu, \nu \in \lx_\X$, if $d_W(\mu,\nu) \leq K$ for  all $W \in
 \X$, then $d_{\lx_\X}(\mu,\nu) \leq K'$.
\end{lemma}

\begin{proof}
  We regard $\mathcal{K}_{\tilde \X}$ as a (twist-free) multicurve
  graph.  For any multicurve $m \in 
  V(\mathcal{K}_{\tilde \X})$, there exists a clean marking 
  $\mu \in V(\lx_\X)$ such that $\base \mu = m$.  In particular 
  every non-annular witness in $\X$ intersects $m$, and if
  $A$ is some annular witness disjoint from $m$ with core curve $a$,
  then $a \in m$ and any complementary
  component adjacent to $a$ is a pair of pants.  Hence we
  may choose a clean transverse curve $b$ for $a$ disjoint from $m
  \setminus \{a\}$.
  Let $\mu$ consist of components $(a,\pi_a b)$ for $a\in m$
  if $a$ is
  parallel to an annular witness, and $(a,\varnothing)$ otherwise.

  By \cite[Prop.~3.6]{vokes}, there exists a path of multicurves 
  $\tilde P\subset \mathcal{K}_{\tilde\X}$ 
  between $\base \mu, \base \nu$ of length $\ell$ 
  at most $K'' = K''(K,\X)$. Let $\tilde P = (m_j)$  
  with $m_1 = \base \mu$ and $m_\ell = \base \nu$.  
  Let $P_0 = (\omega_j) \in V(\lx_\X)$ be a sequence of clean
  markings such that $\omega_1 = \mu, \omega_\ell = \nu$, and
  $\base \omega_j = m_j$, where $\omega_{j \neq 1,\ell}$ is chosen 
  as described above.  
  If $d_{\lx_\X}(\omega_j,\omega_{j+1})$ is uniformly bounded
  indepedent of the path $\tilde P$, then we conclude by the
  triangle inequality.  In fact, it suffices to control projections
  to annular witnesses:

  \begin{claim}
    Let $\omega,\omega' \in V(\lx_\X)$ such that $(\base \omega,
    \base \omega') \in E(\mathcal{K}_{\tilde \X})$, and let
    $D \geq 0$ such that $d_c(\omega,\omega') \leq D$ 
    for any $c \in \base \omega \cup \base \omega'$ parallel to an
    annular witness. Then
    $d_{\lx_\X}(\omega,\omega') \leq (D + 3)(\xi(\Sigma) + 1)$.
  \end{claim}

  \begin{proof}[Proof of claim.]
    Without loss
    of generality assume $|\omega| \leq |\omega'|$, 
    hence $\base \omega'$ is obtained from $\base \omega$ 
    by either adding a component
    $b$ or a twist-free flip $a \mapsto b$.  Let $\omega''$ be 
    obtained from $\omega$ by adding $(b,\trans(\omega',b))$ 
    or replacing 
    $(a,\trans(\omega,a))$ with $(b,\pi_b a)$ and choosing
    a compatible clean marking, as appropriate.  In the first case,
    $d_{\lx_\X}(\omega,\omega'') = 1$, and if $a$ is not parallel
    to an annular witness, then in the second case
    $d_{\lx_\X}(\omega,\omega'')\leq 3$: remove and replace the
    transversal for $a$ with $\pi_a b$, then flip.  Otherwise, $a$
    is a witness curve and at most $d_a(\omega,\omega')\leq D$ 
    twist moves along $a$ suffice to replace $\trans(\omega,a)$ with
    $\pi_a b$, hence $d_{\lx_\X}(\omega,\omega'')\leq D + 1$. 
    
    We show that $d_{\lx_\X}(\omega'',\omega') \leq (D+3)|\omega''|$
    and conclude.
    $\base \omega'' = \base \omega'$, hence to obtain $\omega'$
    we need only modify
    transversals in $\omega''$ to agree with those in
    $\omega'$.  Let $c \in \base \omega''$.  As above, at
    most $2$ moves suffice if $c$ is not a witness curve and 
    $d_c(\omega'',\omega')$ otherwise; 
    by Lemma~\ref{lem:marking_compat},
    $d_c(\omega'',\omega') \leq d_c(\omega,\omega') + 3 \leq D+3$.
  \end{proof}

  Let $P = (\omega_j)$ be a sequence of markings in
  $\lx_\X$ of length $\ell \leq K''$ such that $\omega_1 = \mu$,
  $\omega_\ell = \nu$, $(\base \omega_j)$ is a path in
  $\mathcal{K}_{\tilde \X}$, and $P$ is minimal in the following
  sense: $P$ has the fewest number of pairs
  $(j,a)$ for which $1 \leq j < \ell$, $a \in \base \omega_j \cup
  \base \omega_{j+1}$ is
  parallel to an annular witness, 
  and $d_a(\omega_j,\omega_{j+1}) > K + c_0K'' + 2$, where $c_0$ is
  a universal constant defined below.  $P$
  exists by the existence of $P_0$.  We show that if 
  $P$ has such a pair $(j,a)$ then there exists a more minimal
  sequence $P'$,
  and conclude by contradiction and the above.  Fix $a$ in such a
  pair.
  Let $m_j = \base \omega_j$.  If $a \notin m_j \cup m_{j+1}$, then 
  $d_{a}(\omega_j,\omega_{j+1}) \leq 4$ (see \textit{e.g.}\ the
  proof of \cite[Lem.~2.5]{MMII}).
  For each $j <\ell$ for which $a \in m_j \cup m_{j+1}$, 
  choose $\sigma_{j} \in \Z$ 
  such that $d_a(\omega_j,\tau_{a}^{\sigma_{j}} \omega_{j+1}) 
  \leq 1$, 
  where $\tau_a$ is a Dehn twist about $a$; let $J < \ell$ be the
  last such index. Set $\sigma_{j} = 0$ 
  otherwise. 
  $\cc (a)$ is $\Z$-equivariantly $(1,1)$-quasi-isometric
  to $\Z$ and $d_{a}(\mu,\nu) = d_a(\omega_1,\omega_{\ell})
  \leq K$, hence there exists a universal constant $c_0 > 4$ 
  such that
    \begin{equation}\label{eq:z_bound}
      \left|\sum_{j=1}^J \sigma_{j}\right| 
    \leq K + c_0(\ell-1) + 1\leq K + 
    c_0K''\:.
    \end{equation} 
  
  Let $\rho_j = \sum_{k=1}^{j-1} \sigma_k$ and let 
  $P' = (\omega'_j)$, where $\omega'_j = \tau_a^{\rho_j}\omega_j$ 
  for $j \leq J$, and $\omega'_j = \omega_j$ otherwise; let $m'_j =
  \base \omega'_j$.  We first verify that
  $(m'_j)$ is a path in $\mathcal{K}_{\tilde \X}$.  Since
  $\mathcal{K}_{\tilde \X}$ is preserved by $\tau_a$, for $j < J$ 
  it suffices that 
  $(m_j,\tau_a^{\sigma_j}m_{j+1})$ is an edge in
  $\mathcal{K}_{\tilde \X}$.  For $\sigma_j = 0$, this is immediate.
  Else $a \in m_j \cup m_{j+1}$.  If $a \in m_{j+1}$ then 
  $\tau_a^{\sigma_j}m_{j+1} = m_{j+1}$, and 
  otherwise $a \in m_j$ and
  $m_{j+1}$ is obtained by a twist-free flip $a \mapsto
  b$: we observe that $\tau_a^{\sigma_j} b$ intersects $a$
  minimally in the surface filled by $a,b$. For $j = J$, it suffices
  that $(m_J,\tau_a^{-\rho_J} m_{J+1})$ is an edge in
  $\mathcal{K}_{\tilde\X}$: since $a \in m_J \cup
  m_{J+1}$ by assumption, an identical argument
  applies. 

  We show that $P'$ has
  strictly fewer pairs $(j,a')$ with $a' \in m'_j\cup m'_{j+1}$ a 
  witness
  curve such that 
  $d_{a'}(\omega'_j,\omega'_{j+1}) > K + c_0K'' + 2$.

  \begin{claim}\label{claim:projpreserved} 
    Let $\omega,\omega'$ be clean markings whose respective base
    multicurves $m, m'$ differ by adding a component or a
    (twist-free) flip.
    Let $a \in m \cup m'$ and $a'' \in
    m$ be distinct.  Then $d_{a''}(\omega,\tau_a^q\omega') =
    d_{a''}(\omega,\omega')$ for $q \in \N$.
  \end{claim}

  \begin{proof}[Proof of claim]
If $a'' \in m \cap m'$, then $\pi_{a''}(\omega) =
  \trans(\omega,a'')$ and $\pi_{a''}(\omega') =
    \trans(\omega',a'')$, projections of clean
    transverse curves $b,c$ respectively.  Since $a \in m \cup
    m'$, $a \cap a'' = \varnothing$ and 
    at most one of $b,c$ intersects $a$, hence $\tau_a$ lifts to an
    action on $\cc(a'')$ that
    fixes either $\pi_{a''}(\omega)$ or $\pi_{a''}(\omega')$:
    we obtain $d_{a''}(\omega,\tau_a^q\omega') = 
    d_{a''}(\tau_a^q\omega,\tau_a^{q}\omega')
    = d_{a''}(\omega,\omega')$.  If instead 
    $a'' \in m \setminus m'$, then $a \in m'$, since
    otherwise two distinct curves lie in $m \setminus m'$ and
    $m, m'$ do not differ by adding a component or a flip.
    Then $\pi_{a''}(\omega') = \pi_{a''}(m')$ is preserved by 
    $\tau_a$ and likewise 
    $d_{a''}(\omega,\tau_a^{q}\omega') =
    d_{a''}(\omega,\omega')$. 
  \end{proof}

  Observe that the maps $a' \in m_j' \mapsto \tau_a^{-\rho_j}a'$ 
  and $a' \in m_{j+1}' \mapsto \tau_a^{-\rho_{j+1}}a'$ 
  induce a bijection from $m_j' \cup m_{j+1}'$ to $m_j
  \cup m_{j+1}$ preserving witness curves.  
  We first assume that $j < J$, and suppose that $a' \in m'_j$ is a
  witness curve. Let
  $a'' = \tau_a^{-\rho_{j}}a' \in m_j$.  As
  above, by translating by $\tau_a^{-\rho_{j}}$
  we observe that 
  $d_{a'}(\omega'_j,\omega'_{j+1}) =
  d_{a''}(\omega_j,\tau_a^{\sigma_j} \omega_{j+1})$. If $\sigma_j =
  0$ then $d_{a''}(\omega_j,\tau_a^{\sigma_j} \omega_{j+1}) =
  d_{a''}(\omega_j,\omega_{j+1})$.  
  Suppose $\sigma_j \neq 0$ and thus $a \in m_j \cup m_{j+1}$.  
  For $a \neq
  a''$, Claim~\ref{claim:projpreserved} implies that likewise
  $d_{a''}(\omega_j,\tau_a^{\sigma_j}\omega_{j+1}) =
  d_{a''}(\omega_j,\omega_{j+1})$.   
   An analogous argument applies if $a' \in
  m'_{j+1}$, replacing $a''$ with $\tau_a^{-\rho_{j+1}} a' 
  \in m_{j+1}$ and translating by $\tau_a^{-\rho_{j+1}}$. 
  Finally if $a'' = a$ then $a' = a$, thus 
  by our choice of $\sigma_j$
  $d_{a'}(\omega_j,\tau_a^{\sigma_j}\omega_{j+1}) \leq 1 \leq
  K + c_0K'' + 2$.

  Let $j = J$.  Then $a \in m_J \cup m_{J+1}$
  by our choice of $J$. We apply the arguments above: if $a' \in
  m'_J$, then let $a'' = \tau_a^{-\rho_J}a'\in m_J$ and translate by
  $\tau_a^{-\rho_J}$, else if $a' \in m'_{J+1} = m_{J+1}$, then let
  $a'' =a'$.  Hence if $a'' \neq a$,
  $d_{a'}(\omega_J',\omega_{J+1}') =
  d_{a''}(\omega_J,\omega_{J+1})$.
  If $a'' = a$ then $a' = a$; observe that by \eqref{eq:z_bound}, 
  $|\rho_{J+1}| \leq K + c_0K''$, hence
  $d_a(\omega_{J+1},\tau_a^{\rho_{J+1}}\omega_{J+1}) \leq K + c_0K'' +
  1$.  But $d_a(\omega'_J,\tau_a^{\rho_{J+1}}\omega_{J+1}) =
  d_a(\omega_J,\tau_a^{\sigma_J}\omega_{J+1}) \leq 1$, hence
  $d_{a'}(\omega'_J,\omega'_{J+1}) \leq K + c_0K'' + 2$.
\end{proof}

Hence $(\lx_\X,\X)$ is a hierarchically hyperbolic space with
respect to witness subsurface projections.  However, if
$(X,\mathcal{G})$ is a hierarchically hyperbolic space and $\psi :
X' \to X$ is a quasi-isometry, then $(X',\mathcal{G})$ is likewise
a hierarchically hyperbolic space by precomposing projections with
$\psi$ (this is remarked in \textit{e.g.}\  \cite[\S 1.1.4]{hhsi}).  
Thus by
Section~\ref{sec:qi_marking}, Theorem~\ref{thm:hhsm} follows:
for any $\M$ for which $\X^\M = \X$, the universal 
quasi-isometry $\iota : \M \to \lx_\X$ preserves subsurface
projection, hence 
$(\M,\X)$ is a hierarchically hyperbolic space with
respect to the projections $\pi_W \circ \iota = \pi_W$, $W \in \X$.
\sqed

\section{Hierarchical hyperbolicity of 
multiarc and curve graphs}\label{sec:multiarc}

For an admissible multiarc and curve graph $\mathcal{A}$ on
$\Sigma$, we construct an admissible partial 
marking graph $\M_\A$ on $\Sigma$ with an identical witness
set and a $\PMod(\Sigma)$-equivariant coarse 
quasi-isometry $\zeta: \A \to \M_\A$
that coarsely preserves witness subsurface projection.  Then by
Theorem~\ref{thm:hhsm}, Theorem~\ref{thm:hhsa} follows 
immediately.

\subsection{Constructing the associated marking graph}
Given an arc and curve system $\alpha$, we 
construct a set of corresponding
clean markings $\mu_\alpha$ as follows. Let $\alpha_j$ denote the
connected components of $\alpha$, \textit{i.e.}\  maximal
subsystems such that $\alpha_j \cup \partial \Sigma$ has a single
connected component that is not a boundary. Then $\alpha_j$ 
is either:
\begin{enumerate*}[label={(\roman*)}]
  \item an arc or curve $a_j$, or 
  \item a multiarc $\{a_{j,k}\}_{k=1}^{L_j}$.
\end{enumerate*}
In the former case, let $\base \mu_j = \partial_\Sigma(a)$,
the collection of distinct essential, non-peripheral
boundary components of a regular neighborhood of $a \cup \partial
\Sigma$, and give each component an empty transversal.  
In the latter, let $\base \mu_j = \bigcup_k
\partial_{\Sigma}(\bigcup_{s=1}^k a_{j,s})$ and to each component
$c$ add the transversal $\pi_c (\alpha_j)$. 
Let $F_j$ be the subsurface filled
by $\alpha_j$.  Then $\base \mu_j \setminus \partial F_j\subset
F_j$, $\partial F_j \cap \base \mu_j$ has empty
transversals, and $\mu_j \setminus \partial F_j$ is a complete
marking on $F_j$: $\bigcup_j \mu_j$ is a locally complete marking.  
Let $\mu_\alpha$ be the (finite) collection of all clean markings 
compatible with $\bigcup_j \mu_j$, for every choice of orderings 
of the $a_{j,k}$.  

We define the vertices of $\M_\A$ to be
\[
  V(\M_\A) = \bigcup_{\alpha \in V(\A)}
  \mu_\alpha\:.
\]
Let $(\mu,\nu) \in E(\M_\A)$ if and only if $\mu \in
\mu_\alpha$ and $\nu \in \mu_\beta$ for
$(\alpha,\beta) \in E(\A)$. Finally, define 
$\zeta : V(\A) \to V(\M_\A)$ to be the coarse map
$\alpha \mapsto \mu_\alpha$,
where we observe $\diam(\zeta(\alpha))
\leq 2$ since $\A$ is connected and not a singleton: for
$\beta$ adjacent to $\alpha$ in $\A$,
any $\mu,\mu' \in \zeta(\alpha)$ are adjacent
to $\mu_{\beta}$ in $\M_\A$.  
By our definitions, $\zeta$ is coarsely Lipschitz,
surjective, and since $\mu_\alpha$ is canonical, 
$\PMod(\Sigma)$-equivariant; $\M_\A$ is likewise
preserved under the action of $\PMod(\Sigma)$ and connected.  Let $W
\subset \Sigma$ be an essential, non-pants subsurface.  
For $\alpha = \bigcup_j \alpha_j$ and 
$\mu = \bigcup_j \mu_j \in \mu_\alpha$ as above, $\alpha_j$ 
is filling and $\mu_j\setminus \partial F_j$ is complete on
$F_j$, and $\mu_j \cap \partial F_j$ has empty transversals. Hence
$W$ intersects $\mu$ if and only if it intersects some $F_j$ if and
only if it intersects $\alpha$: the witness set
for $\M_\A$ is identical to that of $\A$.
Finally, by construction $\zeta$ coarsely preserves subsurface
projection to annular witnesses; for non-annular witnesses, 
$\alpha$ and $\mu \in \mu_\alpha$ have uniformly bounded
intersection hence coarsely equal projection.

\subsection{Counting intersections}
It remains to show that if $\mu,\mu'$ are adjacent in $\M_\A$ then
$i(\mu,\mu')$ is uniformly bounded, hence $\M_\A$ is admissible, and
that there exists a Lipschitz coarse retraction for $\zeta$.  If
$\mu,\mu'$ are adjacent then there exists $\alpha,\alpha'$ adjacent
in $\A$ such that $\mu\in\mu_\alpha$ and $\mu' \in \mu_{\alpha'}$.
Since the components of $\mu, \mu'$ are parallel to boundaries and
the components of $\alpha, \alpha'$ respectively,
$i(\base \mu,\base \mu') \leq 4|\mu||\mu'|i(\alpha,\alpha') +
4|\mu||\mu'|(|\alpha| + |\alpha'|) 
\leq 4\xi(\Sigma)^2 (i(\alpha,\alpha')+2\dim \A(\Sigma))$.
Likewise, for $c \in \base \mu$, $d_c(\mu,\mu')$ is uniformly
bounded in terms of $i(\alpha,\alpha')$, and likewise
for $c' \in \base \mu'$: since 
$i(\alpha,\alpha')$ is uniformly bounded, so is $i(\mu,\mu')$.

For $\mu \in V(\M_\A)$, let $E_\mu = \{\alpha \in V(\A) : \mu \in
\mu_\alpha\}$.  By an argument identical to that in
Section~\ref{sec:qi_marking}, $\rho : \mu \mapsto E_\mu$ is a
Lipschitz coarse retraction for $\zeta$ 
if $E_\mu$ has uniformly bounded diameter.  
As usual, it suffices to show that
there exists $D \geq 0$ such that for any $\mu \in V(\M_\A)$ and
$\alpha, \alpha' \in E_\mu$, $i(\alpha,\alpha') \leq D$.
We verify that for any component $a \in \alpha$ and $a' \in
\alpha'$, $i(a,a')$ is bounded uniformly in terms of $\xi(\Sigma)$. 
If $a,a'$ are both
curves, then by construction $a,a' \in \base \mu$
and hence either disjoint or identical. 
If $a$ is an arc and $a'$ a curve,
then $a$ intersects $a'$ at most twice.  More generally,
for any arc $a \in \alpha$, $a$ is contained in a subsurface $F$ for
which $\mu|_F$ is a complete marking on $F$ and $a \setminus \mu$
has at most two components in each component of $F \setminus \mu$.
Up to Dehn twists along components in $\base \mu$, $a$ is thus 
one of finitely many 
arcs, any two of which have at most $8$ intersections in each
component of $F \setminus \mu$. The order of each 
twist is determined up to a uniform constant 
by the transversals in $\mu$.
It follows that any two such arcs have
uniformly bounded intersection number, depending only on $\xi(F)
\leq \xi(\Sigma)$.  Finally, if $a,a'$ are both
arcs then choose $F$ to contain both, whence $i(a,a')$ is uniformly
bounded.
$\zeta$ is a quasi-isometry. \sqed


\section{Acknowledgements}

The author would like to thank Mladen Bestvina for his support and
advice, and Kate Vokes, Jacob Russel, and George Shaji for helpful
discussions.  
The author was supported by NSF award nos.\ 1905720 and 2304774
and no.\ 1840190: \textit{RTG: Algebra, Geometry, and Topology at 
  the University of Utah}.

\bibliographystyle{amsalpha}
\bibliography{pac}

\providecommand{\bysame}{\leavevmode\hbox to3em{\hrulefill}\thinspace}
\providecommand{\MR}{\relax\ifhmode\unskip\space\fi MR }
\providecommand{\MRhref}[2]{%
  \href{http://www.ams.org/mathscinet-getitem?mr=#1}{#2}
}
\providecommand{\href}[2]{#2}
\begin{thebibliography}{BNV22}

\bibitem[BHS17]{hhsi}
Jason Behrstock, Mark~F. Hagen, and Alessandro Sisto, \emph{Hierarchically
  hyperbolic spaces {I}: Curve complexes for cubical groups}, Geometry {\&}
  Topology \textbf{21} (2017), no.~3, 1731--1804.

\bibitem[BHS19]{hhsii}
\bysame, \emph{Hierarchically hyperbolic spaces {II}: Combination theorems and
  the distance formula}, Pacific Journal of Mathematics \textbf{299} (2019),
  no.~2, 257--338.

\bibitem[BHS21]{quasiflats}
\bysame, \emph{Quasiflats in hierarchically hyperbolic spaces}, Duke
  Mathematical Journal \textbf{170} (2021), no.~5, 909 -- 996.

\bibitem[BNV22]{grand}
Assaf Bar-Natan and Yvon Verberne, \emph{The grand arc graph}, 2022.

\bibitem[Kop23]{pac}
Michael~C. Kopreski, \emph{Prescribed arc graphs}, 2023.

\bibitem[MM00]{MMII}
Howard~A. Masur and Yair~N. Minsky, \emph{Geometry of the complex of curves
  {II}: Hierarchical structure}, Geometric \& Functional Analysis \textbf{10}
  (2000), 902--974.

\bibitem[MR20]{mannrafi}
Kathryn Mann and Kasra Rafi, \emph{Large scale geometry of big mapping class
  groups}, 2020.

\bibitem[Sch]{curvenotes}
Saul Schleimer, \emph{Notes on the complex of curves},
  \url{http://homepages.warwick.ac.uk/~masgar/Maths/notes.pdf}.

\bibitem[Sis17]{sisto}
Alessandro Sisto, \emph{What is a hierarchically hyperbolic space?}, 2017.

\bibitem[Vok22]{vokes}
Kate~M. Vokes, \emph{Hierarchical hyperbolicity of graphs of multicurves},
  Algebraic {\&} Geometric Topology \textbf{22} (2022), no.~1, 113--151.

\end{thebibliography}

\end{document}